\documentclass[12pt,bezier]{article}
\usepackage{amssymb}
\usepackage{mathrsfs}
\usepackage{amsmath}
\usepackage{amsfonts,amsthm,amssymb}
\usepackage{amsfonts}
\usepackage{graphics}
\usepackage{graphicx}
\usepackage{color}
\usepackage{ifpdf}
\usepackage{fancyhdr}
\usepackage{epstopdf}
\usepackage{epsfig}
\usepackage{indentfirst}
\usepackage{CJK}
\usepackage{float}
\usepackage{mathtools}

\newtheorem{lem}{Lemma}[section]
\newtheorem{thm}[lem]{Theorem}
\newtheorem{cor}[lem]{Corollary}

\newtheorem{que}[lem]{Question}
\newtheorem{exm}[lem]{Example}

\renewcommand{\thefootnote}{\fnsymbol{footnote}}

\allowdisplaybreaks
\begin{document}
	\title{Upper bounding the coefficients of Tutte polynomials with a variable fixed}
	\author{\CJKfamily{fs}\small  Yuxin Liu$^{a}$\quad Weihua Yang$^{a}$ \quad Xiaxia Guan$^{a*}$\\
			\CJKfamily{kai}\small \textit{$^a$Department of Mathematics,\ Taiyuan University of Technology, P. R. China}\\
		}\date{}
	\maketitle 
	\maketitle
	 
		\begingroup
		\renewcommand{\thefootnote}{}
		\footnotetext{E-mail: guanxiaxia@tyut.edu.cn}
		\endgroup

		{\small{\bf Abstract:}
		The Tutte polynomial is a fundamental invariant of graphs and matroids. For a matroid $M$ of rank $r$, the Tutte polynomial $T_{M}(x,y)$ of $M$ serves as a generalization of the characteristic polynomial $\chi_{M}(x)$, with the relationship given by  $\chi_{M}(x)=(-1)^{r}T_{M}(1-x,0)$. For any real number $t$ satisfying $0\leq t\leq 1$, under suitable exclusion hypotheses on \(M\), we establish an upper bound on the coefficients of the polynomial $T_{M}(x,t)$ (resp., $T_{M}(t,y)$) in terms of  girth $g$ and the number of circuits of size $g$ (resp.,  maximum cardinality $h$ of hyperplanes and the number of hyperplanes of size $h$). The resulting characteristic polynomial bound extends the bound
		of Delaplane and Wu \emph{[Discrete Math.\ 312(12--13) (2012) 2062--2068]} to simple $U_{g-1,g+1}$-free matroids satisfying $n\ge r+c_g$, thereby giving a partial affirmative answer to the question of extending their upper bound beyond binary matroids.

		\section{Introduction}
		
		Matroids were introduced by Whitney \cite{Whitney} in 1935. Let $E$ be a finite set and let $rk: 2^{E}\rightarrow \mathbb{N}$ be a mapping satisfying the following three conditions:
		\begin{itemize}
			\item [(R1)] $0\leq rk(E')\leq |E'| \text{ for any } E'\subseteq E$;
			\item [(R2)] $rk(E')\leq rk(E'') \text{ whenever } E'\subseteq E'' \subseteq E$;
			\item [(R3)]  $rk(E')+ rk(E'')\geq rk(E'\cup E'')+ rk(E'\cap E'') \text{ for any } E', E'' \subseteq E$.
		\end{itemize}
		Define $$\mathcal{I}=\{E'\subseteq E: rk(E')=|E'|\}.$$
		Then $M=(E,rk)$ is called a \emph{matroid}, where $E$ is called the \emph{ground\ set} of $M$, $rk$ is called the \emph{rank\ function} of $M$, and each member in $\mathcal{I}$ is called an \emph{independent\ set} of $M$. A maximal independent set in  $M$ is called a $\mathit{basis}$ or a $\mathit{base}$ of $M$. Let $\mathcal{B}(M)$ denote the set of bases of $M$. An element $e\in E$ is called a \emph{loop} of $M$ if $rk(e)=0$, and a \emph{coloop} of $M$ if $rk(E\setminus \{e\})=rk(E)-1$.
		A non-empty subset $C\subseteq E$ is called a \emph{circuit} of $M$ if $rk(C)=rk(C\setminus \{e\}) =|C|-1=|C\setminus \{e\}|$ for every $e\in C$, that is, a circuit is a nonempty minimal dependent set.
		The \emph{girth} of a matroid is  defined as the length of its shortest circuit, denoted by $g=g(M)$. A matroid $M$ is called \emph{simple} if $g\geq 3$.  For any $ X\subseteq E$, define $ \mathit{cl}(X)=\{x\in E: rk(X\cup x)=rk(X)\}. $
		If $ X=cl(X)$, then $X$ is called a $\mathit{flat}$ (or $\mathit{closed\ set}$) of $M$. We write $r=rk(M)$, and $M$ is called a matroid of rank $r$. A $\mathit{hyperplane}$ of $M$ is a flat with rank $r-1$.

		Let $M$ be a matroid with ground set $E$ and rank function $rk$. The matroids obtained by \emph{deleting} or \emph{contracting}  a subset $A\subseteq E$ from $M$ are denoted by $M\setminus A$ and $M/A$, respectively. For any $X \subseteq E\setminus A$, their rank functions are given by
		$$rk_{M\setminus A}(X)=rk_{M}(X) \text{ and }  rk_{M/A}(X)=rk_{M}(X \cup A)-rk_{M}(A).$$	 
		If $A=\{e\}$, we write $M\setminus \{e\}$ and $M/\{e\}$ as $M\setminus e$ and $M/e$, respectively.
		
		The Tutte polynomial \cite{Tutte} is a fundamental and widely studied invariant in graph theory and matroid theory (extended from graphs to matroids by Crapo \cite{Crapo}). It has wide-ranging applications in combinatorics, statistical physics, knot theory, quantum group theory, and other fields.
		
		For any matroid $ M $ with ground set $E$, the $\mathit{Tutte\ polynomial}$ $ T_{M}(x,y) $ of $ M $ is defined by the following recurrence relation: if $ E = \emptyset $, then $ T_{M}(x,y) =1 $; if $ E \neq  \emptyset $ and $ e\in E $, then
		\begin{equation*}\label{Tutte-DC}
			T_{M}(x,y)=\left\{\begin{array}{ll}
				y\,T_{M\setminus e}(x,y),&\text{if } e \text{ is a loop of } M,\\
				x\,T_{M/ e}(x,y),&\text{if } e \text{ is a coloop of } M,\\
				T_{M\setminus e}(x,y) + T_{M/e}(x,y),&\text{if } e \text{ is neither a loop nor a coloop of } M.
			\end{array}\right.
		\end{equation*}
		
		Throughout the paper, for integers $m$ and $j$, we set
		$\binom{m}{j}=0$ unless $0\le j\le m$.
		Empty sums are interpreted as zero.
		For a polynomial $f$, the notation $[x^k]f$ denotes
		the coefficient of $x^k$ in $f$.
		
		Recently, Ma, Guan and Jin \cite{Ma} established the monotonicity of the coefficients of Tutte polynomials with one variable fixed.
		\begin{lem}\label{lem1.1}\cite{Ma}
			Let $ M_{1} $ and $ M_{2} $ be two matroids on the same ground set such that $\mathcal{B}(M_{1})\subseteq \mathcal{B}(M_{2})$. Then, for any non-negative integer $k$ and any real number $t$ with $0\le t\le 1$, we have
			\begin{align*}
				[x^{k}]T_{M_{1}}(x,t)\le [x^{k}]T_{M_{2}}(x,t) \text{ and } [y^{k}]T_{M_{1}}(t,y)\le [y^{k}]T_{M_{2}}(t,y).
			\end{align*}
		\end{lem}

		Let $n$ and $r$ be two integers with $0\le r\le n$, and let $E$ be a set with $n$ elements. Define the rank function $rk$ as follows: for any $I\subseteq E$, 
		\begin{equation*}
			rk(I)=\left\{\begin{array}{ll}
				|I|,&\text{ if } |I|\le r,\\
				r,&\text{ if } |I|>r.
			\end{array}\right.
		\end{equation*}
		That is, the collection of independent sets is    $\mathcal{I}=\{I\subseteq E\mid |I|\le r\}$.
		The matroid $(E,rk)$ is called a \emph{uniform matroid}, denoted by $U_{r,n}$.
		
		In 2012, Merino,  Ram\'{i}rez-Ib\'{a}\~{n}ez, and Rodr\'{i}guez-S\'{a}nchez \cite{Merino} derived the Tutte polynomial of uniform matroids as follows.
		\begin{lem}\label{lem1.2}\cite{Merino}	
			The Tutte polynomial of $ U_{r,n} $ is the following:
			$$ T_{U_{r,n}}(x,y)=\sum_{i=1}^{r} \binom{n - i - 1}{n - r - 1} x^i + \sum_{j=1}^{n-r} \binom{n - j - 1}{r - 1} y^j $$
			for $0< r< n$. In addition, $T_{U_{n,n}}(x,y)=x^{n} $ and $T_{U_{0,n}}(x,y)=y^{n} $.
		\end{lem}
		
		Note that for any matroid $M$ of rank $r$ on $n$ elements, we have $\mathcal{B}(M)\subseteq \mathcal{B}(U_{r,n})$. Hence,  from Lemmas \ref{lem1.1} and \ref{lem1.2}, we obtain the following result.
		
		\begin{thm}\label{thm1.3}
			Let $ M $ be a matroid of rank $r$ on $n$ elements ($n\geq 1$ and $0\le r\le n$). Then, for any $t$ with $0\leq t\leq 1$ and any non-negative integer
			$k$, we have
			$$ [x^{k}]T_{M}(x,t)\le \binom{n-k-1}{r-k}+\sigma ^{'}, $$
			where $\sigma ^{'}=1$ if $k=r=n$, and $\sigma ^{'}=0$ otherwise.
		\end{thm}

		\begin{proof}
			It is easy to see that the result holds for $r=n$ or $r=0$. Suppose that $0<r<n$. Since $\mathcal{B}(M)\subseteq \mathcal{B}(U_{r,n})$, for $1\le k\le r$, $[x^k]T_{U_{r,n}}(x,t)=\binom{n-k-1}{r-k}$; for $k=0$,  $[x^0]T_{U_{r,n}}(x,t) =\sum_{j=1}^{n-r}\binom{n-j-1}{r-1}t^j
			\le\sum_{j=1}^{n-r}\binom{n-j-1}{r-1}
			=\binom{n-1}{r}$; for $k>r$, both sides are $0$.
		\end{proof}
		
		The \textit{characteristic polynomial} $\chi_M(\lambda)$ of a matroid $M$ of rank $r$, a specialization of the Tutte polynomial, is defined as
		\begin{align*}
			\chi_M(\lambda)=(-1)^{r}T_M(1-\lambda,0).
		\end{align*}
		
		Let $G=(V,E)$ be a connected graph. For any $E'\subseteq E$, let $G[E']$ denote the spanning subgraph with edge set $E'$, that is, $G[E']=(V,E')$. Define $r_{M(G)}(E')=|V|-c(G[E'])$, where $c(G[E'])$ is the number of connected components of $G[E']$. Clearly, $r$ satisfies the three conditions of the rank function of a matroid, which is called the \emph{cycle matroid} of $G$, denoted by $M(G)$. The \textit{characteristic polynomial} of the graphic matroid $M(G)$ is related to the chromatic polynomial $P_{G}(\lambda)$ of a graph $G$ by the equation $P_{G}(\lambda)=\lambda ^{\omega(G)}\chi_{M(G)}(\lambda),$ where $\omega(G)$ is the number of connected components of $G$. 
		In 1978, Li and Tian \cite{Li} obtained an upper bound on the coefficients of the chromatic polynomial of a graph in terms of its girth $g$ and the number of cycles of length $g$. A matroid $M'$ is a \emph{minor} of a matroid $M$ if $M'=(M\setminus A)/B=(M/B)\setminus A$ for two disjoint subsets $A$ and $B$ of the ground set of $M$. A matroid $M$ is called \emph{binary} if it does not contain $U_{2,4}$ as a minor. It is well-known that the graphic matroid of any graph is binary. In 2012, Delaplane and Wu \cite{Delaplane} stated the following bound for the characteristic polynomial of simple binary matroids. Under the binomial convention adopted here, this bound requires a boundary correction, as discussed after Theorem \ref{3.3}.
		\begin{thm}\label{thm1.4}\cite{Delaplane}
			Let $M$ be a simple binary matroid of rank $r$ on $n$ elements with girth $g$. Let $c_{g}$ denote the number of circuits of size $g$. Then for any non-negative integer $k$,
			$$ (-1)^{r-k}[x^{k}]\chi_{M}(x)\le \binom{n}{r-k}-\binom{n-g+2}{r-k-g+2}+\binom{n-c_{g}-g+2}{r-k-g+2}+\sigma , $$ where $\sigma =1$ if $n\le r+c_g$ and $k=r-g+1$, otherwise $\sigma =0$.
		\end{thm}
		
		\textbf{Remark.} There is no correction item $\sigma $ in \cite{Delaplane} for Theorem \ref{thm1.4}. However, correction items are required. Their mistake mainly arose from taking $\binom{m}{0}=0$ instead of $1$ when $m<0$.\begin{exm}
			Let $G=K_{2,5}$ and let $M=M(G)$ be its cycle matroid. Then $n=10$, $r=6$, $g=4$, $c_4=10$, and $[x^4]\chi_M(x)=45$. Under the convention adopted in this paper, namely $\binom{m}{j}=0$ unless $0\le j\le m$, the upper bound in \cite{Delaplane} at $k=r-g+2=4$ evaluates to
			\[
			\binom{10}{2}-\binom{8}{0}+\binom{-2}{0}
			=45-1+0
			=44
			<45=[x^4]\chi_M(x).
			\]
		\end{exm}
				
		They \cite{Delaplane} also showed that this bound fails for $U_{2,n}$ when $n\geq 5$. This motivates the following question.
		\begin{que}\label{non-binary}
			Does the bound in Theorem \ref{thm1.4} hold for other non-binary matroids?
		\end{que}
			
		For a matroid $N$, a matroid $M$ is said to be \emph{$N$-free} if, after deleting any $\lvert E(M)\rvert - \lvert E(N)\rvert$ elements from $M$, the resulting matroid is not isomorphic to $N$.
		In this paper, for any real number $t$ with $0\leq t\leq 1$,  we establish an upper bound on the coefficients of the polynomials $T_{M}(x,t)$ and $T_{M}(t,y)$.

		\begin{thm}\label{thm1.6}
			Let $M$ be a matroid of rank $r$ on $n$ elements with $1\le r\le n-1$ and girth $g$. Let $c_{g}$ denote the number of circuits of size $g$. Let $h$ be the maximum cardinality of the hyperplanes of $M$, and let $H_{h}$ be the number of hyperplanes of size $h$. For any non-negative integer $k$ and any real number $t$ with $0\leq t\leq 1$,
			\begin{enumerate}
				\item [(i)] if $M$ is simple and $U_{g-1,g+1}$-free, then	\begin{flalign*}
					&
					&	[x^k]T_M(x,t)
					\leq&
					\binom{n-k-1}{r-k}
					-
					\binom{n-k-g+1}{r-k-g+2}
					+
					\binom{n-c_g-k-g+1}{r-k-g+2}
					&&\\
					&&
					&+t
					\left[
					\binom{n-k-g}{r-k-g+2}
					-
					\binom{n-c_g-k-g}{r-k-g+2}
					\right]
					+\sigma,
					&&
				\end{flalign*}
				where
				\[
				\sigma = 
				\begin{cases}
					t, & \text{if } n \leq r + c_g \text{ and } k = r - g + 1, \\
					1, & \text{if } n \leq r + c_g \text{ and } k = r - g + 2, \\
					0, & \text{otherwise}.
				\end{cases}
				\]
				\item [(ii)]  if \(h\le n-3\) and \(M\) has no contraction isomorphic to \(U_{2,n-h+1}\), then	
				\begin{flalign*}
					&&
					\begin{aligned}[t]
						[y^k]T_M(t,y)
						\leq&
						\binom{n-k-1}{r-1}
						-
						\binom{h-k+1}{r-1}
						+
						\binom{h-k-H_h+1}{r-H_h-1}
						\\
						&+t
						\left[
						\binom{h-k}{r-2}
						-
						\binom{h-k-H_h}{r-H_h-2}
						\right]
						+\theta,
					\end{aligned}
					&&
				\end{flalign*}
				where
				\[
				\theta = 
				\begin{cases}
					t, & \text{if } r \le H_h \text{ and } k = h - r + 1, \\
					1, & \text{if } r \le H_h \text{ and } k = h - r + 2, \\
					0, & \text{otherwise}.
				\end{cases}
				\]
			\end{enumerate}
		\end{thm}
		
		Theorem \ref{thm1.6} gives upper bounds complementary to the uniform-matroid bound in Theorem \ref{thm1.3}. Moreover, if $M$ is a simple binary matroid, then $M$ must be $U_{g-1,g+1}$-free. Theorem \ref{3.3} extends the bound in Theorem \ref{thm1.4} to simple $U_{g-1,g+1}$-free matroids satisfying $n\ge r+c_g$ and answers Question \ref{non-binary} partly.
		
		 Theorem \ref{thm1.6} is not necessarily true for other matroids.    Remarkably, this result holds for all matroids when $t=1$.
		
		\begin{thm}\label{thm1.7}
			Let $M$ be a matroid of rank $r$ on $n$ elements with
			$1\le r\le n-1$ and girth $g$. Let $ c_{g} $ denote the number of circuits of size $ g $. Let $h$ be the maximum cardinality of the hyperplanes of $M$, and let $H_{h}$ be the number of hyperplanes of size $h$. Then, for any non-negative integer $ k $,
			\begin{flalign*}
				&(i)\qquad
				&&
				[x^k]T_M(x,1)
				\leq
				\binom{n-k-1}{r-k}
				-
				\binom{n-k-g}{r-k-g+1}
				+
				\binom{n-c_g-k-g}{r-k-g+1}
				+\sigma,
				&&
			\end{flalign*}
		where $\sigma=1$ if $n\le r+c_g$ and $k\in\{r-g+1,r-g+2\}$, and $\sigma=0$ otherwise.
			\begin{flalign*}
				&(ii)\qquad
				&&
				[y^k]T_M(1,y)
				\leq
				\binom{n-k-1}{r-1}
				-
				\binom{h-k}{r-1}
				+
				\binom{h-k-H_h}{r-H_h-1}
				+\theta,
				&&
			\end{flalign*}
			where $\theta=1$ if $r\le H_h$ and $k\in\{h-r+1,h-r+2\}$, and $\theta=0$ otherwise.
			
		\end{thm}
		
		\section{Proof of Theorems \ref{thm1.6} and \ref{thm1.7}}
		
		This section is devoted to proving Theorems \ref{thm1.6} and \ref{thm1.7}. The following elementary lemmas will be needed. We begin with a binomial-coefficient inequality, which can be readily verified using properties of binomial coefficients.
		\begin{lem}\label{lem2.1}\cite{Li}
			For non-negative integers $m_{1}$, $m_{2}$, $x$, $t$ with $m_{1}\geq m_{2}$ and $x\leq m_{1}-m_{2}$, we have
			\begin{equation}\label{combinatorial-identity}
				\binom{m_{1}-x}{t} + \binom{m_{2}+x}{t} \le \binom{m_{1}}{t} + \binom{m_{2}}{t}.
			\end{equation}	
			Moreover, Inequality (\ref{combinatorial-identity}) also holds for integers $m_1,m_2,x,t$ satisfying $m_1\ge0$, $m_2<0$, $0\le x\le m_1-m_2$,
			and $t\ge1$.
		\end{lem}
		\begin{proof}
			The result has been proved in \cite{Li} for $m_2\ge 0$.
			
			Now suppose that \(m_2<0\). If \(m_2+x\le0\) or $m_1-x\le 0$, then $\binom{m_2+x}{t}=\binom{m_2}{t}=0$ or $\binom{m_1-x}{t}=0$. Since \(m_1-x\le m_1\) and $m_2+x\le m_1$, $\binom{m_1-x}{t} \le \binom{m_1}{t}$ and $\binom{m_2+x}{t}\le \binom{m_1}{t}$, and the result follows. It remains to consider the $m_1-x>0$, $m_2+x>0$. Set $a=m_1-x$ and $b=m_2+x$. Then \(a,b>0\). By Vandermonde's identity $\binom{a+b}{t} = \sum_{i=0}^{t}\binom ai\binom b{t-i}$, we have $$\binom at+\binom bt \le \binom{a+b}{t}.$$ Note that $a+b=m_1+m_2<m_1$. Then $\binom{a+b}{t} \le \binom{m_1}{t}$. Therefore,
			\[
			\begin{aligned}
				\binom{m_1-x}{t}
				+\binom{m_2+x}{t}
				=
				\binom at+\binom bt
				\le
				\binom{a+b}{t}
				\le
				\binom{m_1}{t}
				=
				\binom{m_1}{t}+\binom{m_2}{t}.
			\end{aligned}
			\]
			This proves the lemma.
		\end{proof}
			
			We use the convention that
			\[
			\binom{m}{0}=
			\begin{cases}
				1,&m\geq0,\\
				0,&m<0.
			\end{cases}
			\]
			Then the following result holds.
			
		\begin{lem}\label{combinatorial-identity2}
			For integers $m_{1}$, $m_{2}$, $x$ with $m_{1}\geq 0$, $m_{1}\geq m_{2}$ and $0\le  x\le  m_{1}-m_{2}$, and for any $t$ with $0\le t \le 1$, we have
			\begin{equation}
				\begin{aligned}
					&\binom{m_1-x}{0}+\sigma_1
					+\binom{m_2+x}{0}+\sigma_2
					+(1-t)\left[
					\binom{m_1-x}{1}+\binom{m_2+x}{1}
					\right]\\
					\le{}&
					\binom{m_1}{0}+\binom{m_2}{0}+\sigma
					+(1-t)\left[
					\binom{m_1}{1}+\binom{m_2}{1}
					\right],
				\end{aligned}
			\end{equation}
		where $\sigma_1=t$ if $m_1-x<0$ and $\sigma_1=0$ otherwise;
		$\sigma_2=t$ if $m_2+x<0$ and $\sigma_2=0$ otherwise;
		and $\sigma=t$ if $m_2<0$ and $\sigma=0$ otherwise.
		\end{lem} 
		\begin{proof}
		For every integer $a$,
				\[
				\binom{a+1}{1}-\binom{a}{1}=\binom{a}{0}.
				\]
				Then, by the definitions of $\sigma_1,\sigma_2,\sigma$ and the assumption $m_1\ge0$, the inequality is equivalent to
				\[
				\begin{aligned}
					(1-t)\left[
					\binom{m_1-x+1}{1}+\binom{m_2+x+1}{1}
					\right] \le (1-t)\left[
					\binom{m_1+1}{1}+\binom{m_2+1}{1}
					\right].
				\end{aligned}
				\]
				By Lemma \ref{lem2.1}, we have
				\[
				\binom{m_1-x+1}{1}+\binom{m_2+x+1}{1}
				\le
				\binom{m_1+1}{1}+\binom{m_2+1}{1}.
				\]
				Multiplying both sides by $1-t\ge0$ completes the proof.
			\end{proof}
			We now consider some simple results for matroids.
		\begin{lem}\label{lem 2.4}
			Let $M$ be a matroid of rank $r$ on $n$ elements with girth $g$. For any $e\in E$, if $M$ is $U_{g-1,g+1}$-free, then $M\setminus e$ is $U_{g-1,g+1}$-free.
		\end{lem}
		\begin{lem}\label{lem2.2} 
			Let $M$ be a matroid on ground set $E$ with girth $g\ge 2$. Assume that $ e $ is an element contained in a circuit of size $g$. Then the girth $g(M/e)$ of $M/e$ is equal to $g-1$. Moreover, 
			\begin{eqnarray*}
				&&\{C\subseteq E: e\in C, C\setminus \{e\} \text{ is a circuit of } M/e, \text{ and } |C\setminus \{e\}|=g(M/e)\}\\
				&=&\{C'\subseteq E: e\in C', C' \text{ is a circuit of } M, \text{ and } |C'|=g\}.
			\end{eqnarray*}
			In particular, the number of $(g-1)$-circuits in $M/e$ is equal to the number of \(g\)-circuits of \(M\) containing \(e\).
		\end{lem}
		\begin{proof}
			Let $rk$ denote the rank function of $M$.  Clearly, $e$ is not a loop in $M$ as $g\ge 2$. Then by the definition of contraction, for any $ X\subseteq E\setminus \{e\} $,$$	rk_{M/e}(X) = rk_{M}(X \cup \{e\}) - rk_{M}(e)=rk_{M}(X \cup \{e\})-1.$$
			
			If $C'$ is a circuit of size $g$ containing $e$ in $M$, then $rk_M(C')=|C'|-1 =rk_M(C'\setminus\{f\})$ for every $f\in C'$. This implies that $rk_{M/e}(C'\setminus\{e\}) =|C'|-2 =rk_{M/e}(C'\setminus\{e,f\})$ for every $f\in C'\setminus\{e\}$. Hence, $C'\setminus\{e\}$ is a circuit in $M/e$, and $g(M/e)\le g-1$.
			
			Let $C\setminus \{e\}$ be a circuit of size $g(M/e)$ in $M/e$, that is, $rk_{M/e}(C\setminus \{e\})=|C\setminus \{e\}|-1=|C|-2=g(M/e)-1$. Then $rk_{M}(C)=rk_{M/e}(C\setminus \{e\})+1=|C|-1$. This implies that $C $ contains a circuit $C'$ in $M$. Since the girth of $M$ is $g$, we have $g\leq |C'|\leq |C|=g(M/e)+1$. This implies $g(M/e)\geq g-1$.  Hence, $g(M/e)=g-1$, and  $C'=C$ is a circuit of $M$. 
		\end{proof}
		
		\begin{lem}\label{lem 2.3}
			Let $M$ be a matroid of rank $r$ on $n$ elements with girth $g\ge 2$. Let $e$ be an element of a $g$-circuit. If $M$ is $U_{g-1,g+1}$-free, then $M/e$ is $U_{g-2,g}$-free.
		\end{lem}
		\begin{proof}
			Assume that \(M/e\) contains a
			\(U_{g-2,g}\), where the ground set of \(U_{g-2,g}\) is \(E'\). Then $|E'|=g$ and $rk_{M/e}(E')=rk_{U_{g-2,g}}(E')=g-2$ by the definition of deletion. By the definition of uniform matroids, $g(U_{g-2,g})=g-1$, and $E'\setminus \{f\}$ is a circuit of size $g-1$ of $M/e$ for any \(f\in E'\). By Lemma \ref{lem2.2}, \(E'\setminus \{f\}\cup\{e\}\) is a circuit of \(M\). Then \begin{equation}\label{circult-rank}
				rk(E'\setminus \{f\}\cup\{e\})=rk(E'\setminus \{f\})=|E'|-1=g-1.
			\end{equation} 
			By (R3), $	rk(E')+rk(E'\setminus \{f\}\cup\{e\})
			\geq
			rk(E'\setminus \{f\})+rk(E'\cup\{e\}).$
			Then $rk(E')\geq rk(E'\cup \{e\}).$ On the other hand, by (R2) and the definition of contraction, $$rk(E')= rk(E'\cup\{e\})=rk_{M/e}(E')+1=g-1.$$ 
			By Equation (\ref{circult-rank}), \(E'\) is a circuit of \(M\). Hence, every \(g\)-subset of \(E'\cup\{e\}\) is a circuit of \(M\), that is, $M\setminus (E(M)\setminus E'\setminus e)\cong U_{g-1,g+1}$, contrary to the assumption that \(M\) is \(U_{g-1,g+1}\)-free.
		\end{proof}

		\begin{lem}\label{unique-circuit}
			Let $M$ be a matroid on $n$ elements of rank $r$ with girth $g$. If $r=n-1$, then for any non-negative integer $k$ and any real number $t$ with $0\leq t\leq 1$,
			\begin{eqnarray*}
				[x^{k}]T_{M}(x,t) &=& \binom{n-k-1}{n-k-1} - \binom{n-k-g+1}{n-k-g+1}+\sigma,
			\end{eqnarray*}
			where $\sigma=t$ if $k=n-g$, $\sigma=1$ if $k=n-g+1$ and $\sigma=0$ otherwise.
		\end{lem}
		\begin{proof}
			Since $r=n-1$, we have that $M$ has a unique circuit $C$ with $|C|=g$, and there are $n-g$ coloops in $M$. A straightforward verification shows that
			\begin{eqnarray*}
				T_{M}(x,y)&=&\sum_{i=n-g+1}^{n-1}x^i+x^{n-g}y.
			\end{eqnarray*}
			Hence,
			\begin{equation*}\label{Tutte-DC}
				[x^{k}]T_{M}(x,t)=\left\{\begin{array}{ll}
					1,&n-g+1\leq k\leq n-1,\\
					t,&k=n-g=r-g+1,\\
					0,&\text{ otherwise}.
				\end{array}\right.
			\end{equation*}
			
			Let $A=\binom{n-k-1}{n-k-1} - \binom{n-k-g+1}{n-k-g+1}$.
			\begin{itemize}
				\item [(i)] If $ n-g+1< k\leq n-1 $, then $A=1-0=1$.
				\item [(ii)] If $k>n-1$ or $k\le n-g+1$, then $A=0$,
				except when $g=1$ and $k=n$.
			\end{itemize}
			In each case, $A+\sigma=[x^k]T_M(x,t)$.
		\end{proof}

		The $\mathit{dual}$ of a matroid $M = (E,r)$ is defined by the matroid $M^{*} = (E,r^{*})$, where $r^{*}(A) = |A| + r(E\setminus A)-r$ for any $A\subseteq E$. Then 
		\begin{eqnarray}\label{dual-rank}
			r^{*}=r^{*}(M^{*})=|E|-r,
		\end{eqnarray}
		$(M^{*})^{*}=M$, and $e\in E$ is a loop in $M$ if and only if $e$ is a coloop in $M^{*}$. From the definition of the Tutte polynomial, it follows that 
		\begin{eqnarray}\label{dual-Tutte}
			T_{M} (x,y) = T_{M^{*}}(y,x).
		\end{eqnarray}
		For any $X\subseteq E$, $r(X\cup e)=r(X)+1=r$ if and only if $r^{*}(E\setminus X\setminus e)=r^{*}(E\setminus X)=|E\setminus X|-1$ for all $e\in E\setminus X$. Hence,
		\begin{eqnarray}\label{dual-circult}
			\text{$X$ is a circuit of $M^*$ if and only if $E \setminus X$ is a hyperplane of $M$.} 
		\end{eqnarray} 
		\textbf{Proof of Theorems \ref{thm1.6} and \ref{thm1.7}}.
		
		For Theorem \ref{thm1.6}(ii), the assumption $h\le n-3$ implies that $M^{*}$ is simple. Moreover, since $(M^{*}\setminus A)^{*}=M/A$ for any $A\subseteq E$, the forbidden-contraction hypothesis implies that $M^{*}$ is $U_{n-h-1,n-h+1}$-free. Thus (ii) follows by applying (i) to $M^{*}$. We only need to prove conclusion (i) using (\ref{dual-rank}), (\ref{dual-Tutte}), and (\ref{dual-circult}). At $t=1$, the expression in Theorem \ref{thm1.6}(i) equals that in Theorem \ref{thm1.7}(i) for $k\le r-g+1$. For $k\ge r-g+2$, Theorem \ref{thm1.7}(i) follows directly from Theorem \ref{thm1.3}. Assume that $r=n-1$. Then  $c_{g}=1$. This implies that
		\begin{eqnarray*}
			&&\binom{n-k-1}{n-1-k} - \binom{n-k-g+1}{n-1-k-g+2} + \binom{n-1-k-g+1}{n-1-k-g+2}\\&+& t\left[\binom{n-k-g}{n-1-k-g+2} - \binom{n-1-k-g}{n-1-k-g+2} \right]\\&=&\binom{n-k-1}{n-k-1} - \binom{n-k-g+1}{n-k-g+1}.
		\end{eqnarray*}
		Conclusion (i) holds by Lemma \ref{unique-circuit}.
		
		Now suppose that $r\leq n-2$. If $k>r-g+2$, then $$\binom{n-k-g+1}{r-k-g+2}=\binom{n-c_{g}-k-g+1}{r-k-g+2}=\binom{n-c_{g}-k-g}{r-k-g+2}=\binom{n-k-g}{r-k-g+2}=0.$$
		In this case, the conclusion also holds by Theorem \ref{thm1.3}.

		Suppose that $k=r-g+2$. Note that $r\leq n-2$. Then $n-k-g\geq 0$. This implies 
		\begin{eqnarray*}
			B	&=&
			\binom{n-k-1}{r-k}-\binom{n-k-g+1}{r-k-g+2}	+\binom{n-c_g-k-g+1}{r-k-g+2}\\
			&&+t
			\left[
			\binom{n-k-g}{r-k-g+2}
			-
			\binom{n-c_g-k-g}{r-k-g+2}
			\right]\\
			&=&\binom{n-k-1}{r-k}-1	+\binom{n-c_g-k-g+1}{0}+t
			\left[1-\binom{n-c_g-k-g}{0}
			\right].
		\end{eqnarray*}
		
		\begin{itemize}
			\item [(i)] If $c_g\leq n-k-g$, then $B=\binom{n-k-1}{r-k}$.
			\item [(ii)] If $c_g\geq n-k-g+2=n-r$, then $B=\binom{n-k-1}{r-k}-1+t$.
			\item [(iii)] If $c_g=n-k-g+1$, then $B=\binom{n-k-1}{r-k}+t$.
		\end{itemize}
		In either case, the upper bound is satisfied by Theorem \ref{thm1.3}.
		
		We now suppose that $k\leq r-g+1\leq n-g-1$. We use induction on the girth $g$, and, for each fixed $g$, induction on $c_g$. The base girth is $1$ when $t=1$, and $3$ when $0\le t<1$. Thus the induction step assumes $g\ge2$ when $t=1$, and $g\ge4$ otherwise.  We begin with the induction step.
		
		\textbf{Induction step.}  Assume inductively that the assertion holds for every admissible matroid of girth less than \(g\) with $g\ge 2$. Let $e$ be an element of a circuit of size $g$ in $M$. Then $e$ is neither loop or coloop. By Lemma \ref{lem2.2}, $g(M/e)=g-1$. Let $c_{g-1}(M/e)$ denote the number of circuits of size $g-1$ in $M/e$. By Lemma \ref{lem2.2} again, $c_{g-1}(M/e)=d_{g}(e)$.  First, we examine the case where $c_{g}=1$. By Theorem \ref{thm1.3} and $r(M\setminus e)=r\le (n-1)-1$, we have
		\begin{eqnarray*}
			[x^{k}]T_{M\setminus e}(x,t) &\le& \binom{(n-1)-k-1}{r-k}=\binom{n-k-2}{r-k} .
		\end{eqnarray*}
		Clearly, $c_{g-1}(M/e)=1<(n-1)-(r-1)=|E(M/e)|-r(M/e)$ as $r\le n-2$. By the inductive hypothesis and Lemma \ref{lem 2.3},
		\begin{eqnarray*}
			[x^{k}]T_{M/e}(x,t)&\le&\binom{n-k-2}{r-k-1} - \binom{n-k-g+1}{r-k-g+2} + \binom{n-k-g}{r-k-g+2}\\
			&&+t\left[ \binom{n-k-g}{r-k-g+2} - \binom{n-k-g-1}{r-k-g+2} \right].
		\end{eqnarray*}
		Hence, by the definition of the Tutte polynomial,
		\begin{eqnarray*}
			[x^{k}]T_{M}(x,t) &=& [x^{k}]T_{M\setminus e}(x,t) +[x^{k}]T_{M/e}(x,t)\\ &\le& \binom{n-k-1}{r-k} - \binom{n-k-g+1}{r-k-g+2} + \binom{n-k-g}{r-k-g+2}\\&&+t\left[\binom{n-k-g}{r-k-g+2} - \binom{n-k-g-1}{r-k-g+2}  \right].
		\end{eqnarray*}
		Hence the inequality holds when $c_{g}=1$.
		
		Now suppose that the result holds for all admissible matroids of girth $g$ with fewer than $ c_{g} $ circuits of size $g$. Then there exist two circuits in $M$ $C_1$ and $C_2$ satisfying $C_1\neq C_2$ and $|C_1|=|C_2|=g$. This implies that $C_2$ is also a circuit in $M\setminus e'$  for any $e'\in C_1\setminus C_2$. Let $e$ be an element of a circuit of size $g$ in $ M $ such that $ g(M\setminus e)=g$. Let $c_{g}(M\setminus e)$ denote the number of circuits of size $g$ in $ M\setminus e$. Let $d_{g}(e)$ denote the number of circuits containing $e$ of size $g$ in $M$.  Note that $ c_{g}(M\setminus e) =c_{g}-d_{g}(e) $. Then, by Lemmas \ref{lem 2.4} and \ref{lem 2.3}, and the inductive hypothesis (\textcolor{red}{only by the inductive hypothesis for  $t=1$}),
		\begin{eqnarray*}
			[x^{k}]T_{M/e}(x,t)&\le& \binom{n-k-2}{r-k-1} - \binom{n-k-g+1}{r-k-g+2} + \binom{n-d_{g}(e)-k-g+1}{r-k-g+2}\\&&+t\left[ \binom{n-k-g}{r-k-g+2} - \binom{n-d_{g}(e)-k-g}{r-k-g+2}  \right] +\sigma ',
		\end{eqnarray*}
		where $\sigma '=t$ if $n\leq r+d_{g}(e)$ and $k=r-g+1$, $\sigma '=0$ otherwise,
		and	
		\begin{eqnarray*}
			[x^{k}]T_{M\setminus e}(x,t)&\le& \binom{n-k-2}{r-k} - \binom{n-k-g}{r-k-g+2}+ \binom{n-c_{g}+d_{g}(e)-k-g}{r-k-g+2}\\
			&&+t\left[ \binom{n-k-g-1}{r-k-g+2} - \binom{n-c_{g}+d_{g}(e)-k-g-1}{r-k-g+2} \right]+\zeta,
		\end{eqnarray*}
		where $\zeta=t$ if $n\leq r+c_{g}-d_{g}(e)+1$ and $k=r-g+1$, $\zeta=0$ otherwise. Denote 
		\begin{eqnarray*}
			A&=&\left[ \binom{n-d_{g}(e)-k-g+1}{r-k-g+2} +\binom{n-c_{g}+d_{g}(e)-k-g}{r-k-g+2} \right]\\
			&&-t \left[ \binom{n-d_{g}(e)-k-g}{r-k-g+2} +\binom{n-c_{g}+d_{g}(e)-k-g-1}{r-k-g+2} \right]
		\end{eqnarray*}
		and
		\begin{eqnarray*}
			B=\left[\binom{n-k-g}{r-k-g+2}+\binom{n-c_{g}-k-g+1}{r-k-g+2} \right]-t \left[ \binom{n-k-g-1}{r-k-g+2} +\binom{n-c_{g}-k-g}{r-k-g+2} \right].
		\end{eqnarray*}
		If $k=r-g+1$, then by Lemma \ref{combinatorial-identity2}, $A+\sigma ' +\zeta =B+ \sigma $.
		If $k\neq r-g+1$ (i.e., $k<r-g+1$), then by Lemma \ref{lem2.1} (setting $ x=d_{g}(e)-1$), we have
		\begin{eqnarray*}
			A&=&   \left[ \binom{n-d_{g}(e)-k-g}{r-k-g+1} +\binom{n-c_{g}+d_{g}(e)-k-g-1}{r-k-g+1} \right]\\
			&&+(1-t) \left[ \binom{n-d_{g}(e)-k-g}{r-k-g+2} +\binom{n-c_{g}+d_{g}(e)-k-g-1}{r-k-g+2} \right]\\
			&\le&  \left[\binom{n-k-g-1}{r-k-g+1}+\binom{n-c_{g}-k-g}{r-k-g+1} \right]\\
			&&+ (1-t) \left[ \binom{n-k-g-1}{r-k-g+2} +\binom{n-c_{g}-k-g}{r-k-g+2} \right]   \quad \text{(by Lemma 2.1)} \\ 
			&=&B.
		\end{eqnarray*}
		By the definition of the Tutte polynomial, we have
		\begin{eqnarray*}
			[x^{k}]T_{M}(x,t) &=& [x^{k}]T_{M\setminus e}(x,t) +[x^{k}]T_{M/e}(x,t) \\
			&\le & \left[ \binom{n-k-2}{r-k-1} +\binom{n-k-2}{r-k} \right] - \binom{n-k-g+1}{r-k-g+2} -\binom{n-k-g}{r-k-g+2}\\
			&&+  \left[ \binom{n-d_{g}(e)-k-g+1}{r-k-g+2} +\binom{n-c_{g}+d_{g}(e)-k-g}{r-k-g+2} \right]\\
			&& +t\left[ \binom{n-k-g}{r-k-g+2} +\binom{n-k-g-1}{r-k-g+2} \right] \\ 
			&&-t \left[ \binom{n-d_{g}(e)-k-g}{r-k-g+2} +\binom{n-c_{g}+d_{g}(e)-k-g-1}{r-k-g+2} \right] +\sigma '+\zeta \\ 
			&\le& \binom{n-k-1}{r-k} - \binom{n-k-g+1}{r-k-g+2} + \binom{n-c_{g}-k-g+1}{r-k-g+2}\\
			&&+t\left[ \binom{n-k-g}{r-k-g+2} - \binom{n-c_{g}-k-g}{r-k-g+2}  \right] +\sigma,
		\end{eqnarray*}
		where $\sigma=t$ if $n\leq r+c_{g}$ and $k=r-g+1$, $\sigma=0$ otherwise.

			\textbf{Base step.} (a) For	$t=1$, it suffices to prove that if $g=1$, $r\leq n-2$ and $k\leq r-g+1=r$, then
			\[
			[x^k]T_M(x,1)
			\le
			\binom{n-c_1-k-1}{r-k}+\sigma,
			\]
			where $\sigma=1$ if $n\le r+c_1$ and $k=r$,
			and $\sigma=0$ otherwise.
			
			We now proceed by induction on the number  $ c_{1} $ of loops. Let $ e $ be a loop of $ M $. Note that $r(M\setminus e)=r$ and $|E(M\setminus e)|=n-1$. First, assume that
			$c_{1} = 1$. Then, by Theorem \ref{thm1.3},
			$$ [x^{k}]T_{M}(x,1) = [x^{k}]T_{M\setminus e}(x,1) \le \binom{n-k-2}{r-k}=\binom{n-c_{1}-k-1}{r-k}.$$ 
			
			Thus, the inequality holds when $g=1$ and $c_{1}=1$.
			
			Now assume that the result holds for all matroids of girth $ g=1 $ containing fewer than $ c_{1} $ loops. Note that $ r(M\setminus e) = r$. Let $ c_{1}(M\setminus e) $ denote the number of loops in $ M\setminus e $.  It is easy to verify $c_{1}(M\setminus e)=c_{1}-1$.
			Then by the inductive hypothesis,
			\begin{eqnarray*}
				[x^{k}]T_{M}(x,1)=[x^{k}]T_{M\setminus e}(x,1)&\le&\binom{(n-1)-(c_{1}-1)-k-1}{r-k}+\sigma\\
				&=&\binom{n-c_{1}-k-1}{r-k}+\sigma,
			\end{eqnarray*}
			where $\sigma=1$ if $n\leq r+c_g$ and $k=r$, and $\sigma=0$ otherwise.	Hence, the conclusion holds. This completes the proof of Theorem \ref{thm1.7}.

				(b)	For	$t=0$, it suffices to prove that if $g=3$, $r\leq n-2$ and $k\leq r-g+1=r-2$, then
				\begin{flalign*}
					[x^k]T_M(x,0)
					\leq
					\binom{n-k-2}{r-k}
					+
					\binom{n-c_3-k-2}{r-k-1}.
				\end{flalign*}
				We now use induction on $c_{3}$.
				First assume that $c_{3}=1$, and let $e$ be an element of this $3$-circuit. By Theorem  \ref{thm1.3},
				$$ [x^{k}]T_{M\setminus e}(x,0) \le \binom{n-k-2}{r-k}.$$
				By Lemma \ref{lem2.2}, $g(M/e)= 2$. 
				Now let $f$ be an element of a circuit of size $2$ in $M/e$. By Lemma \ref{lem2.2}, we have $g(M/e/f)=1$. Let $z$ be a loop in $M/e/f$.
				Note that $|E(M/e\setminus f)|=n-2$, $rk(M/e\setminus f)=r-1\le (n-2)-1$, $|E(M/e/f\setminus z)|=n-3$, and $rk(M/e/f\setminus z)=r-2\le (n-3)-1$.
				Then by Theorem \ref{thm1.3} and the definition of the Tutte polynomial,
				\begin{eqnarray*}
					[x^{k}]T_{M/e}(x,0)&=&[x^{k}]T_{M/e\setminus f}(x,0) + [x^{k}]T_{M/e/f}(x,0)\\
					&\le&\binom{n-k-3}{r-k-1} + 0[x^{k}]T_{M/e/f\setminus z}(x,t)\\ &= &\binom{n-k-3}{r-k-1}.
				\end{eqnarray*}
				Hence, by the definition of the Tutte polynomial,
				\begin{eqnarray*}
					[x^{k}]T_{M}(x,0)&=&[x^{k}]T_{M\setminus e}(x,0) + [x^{k}]T_{M/e}(x,0)\\
					&\le &\binom{n-k-2}{r-k} + \binom{n-k-3}{r-k-1}.
				\end{eqnarray*}
				So the inequality is satisfied when $t=0$, $g=3$ and $c_{3}=1$.
				
				Now assume that the result holds for all admissible matroids of girth $3$ with fewer than $c_{3}$ circuits of size $3$. Let $e$ be an element of a $3$-circuit such that $g(M\setminus e)=3$. By Lemmas \ref{lem 2.4} and \ref{lem 2.3}, $M/e$ and $M\setminus e$ are
				$U_{1,3}$-free and $U_{2,4}$-free, respectively,
				since $M$ is $U_{2,4}$-free.
				By Lemma \ref{lem2.2}, $g(M/e)=2$.
				Let $f$ be an element of a circuit of size $2$ in $M/e$. After contracting $f$, we obtain a loop by Lemma \ref{lem2.2}. More precisely, there exists a unique loop in $M/e/f$ as $M/e$ is $U_{1,3}$-free. Let $d_{3}(e)$ denote the number of circuits of size $3$ which contain $e$ in $M$. Let $f_{1}, f_{2}, \dots , f_{d_{3}(e)}$ be the elements of circuits of size $2$ in $M/e$ such that $\{f_{i}, f_{j}\}$ is not a circuit in $M/e$ for any  $1\le i<j\le d_{3}(e)$. Note that $k\leq r-2$. Then by Theorem \ref{thm1.3} and the definition of the Tutte polynomial,
				\begin{eqnarray*}
					[x^{k}]T_{M/e}(x,0)&=&[x^{k}]T_{M/e\setminus f_{1}}(x,0) + [x^{k}]T_{M/e/f_{1}}(x,0)\\
					&=&[x^{k}]T_{(M/e\setminus f_{1})\setminus f_{2}}(x,0) + [x^{k}]T_{(M/e\setminus f_{1})/f_{2}}(x,0) +0\binom{n-k-4}{r-k-2}\\
					&=& [x^{k}]T_{(M/e\setminus f_{1}\setminus f_{2})\setminus f_{3}}(x,0) + [x^{k}]T_{(M/e\setminus f_{1}\setminus f_{2})/f_{3}}(x,0) +0\binom{n-k-5}{r-k-2}\\
					&=& \dots =[x^{k}]T_{M/e\setminus \{f_{1},f_{2},\dots,f_{d_{3}(e)}\}}(x,0) \\
					&\leq & \binom{n-k-2-d_{3}(e)}{r-k-1}.
				\end{eqnarray*}
				
				Note that $c_{3}(M\setminus e)=c_{3}-d_{3}(e)$, where $c_{3}(M\setminus e)$ denotes the number of circuits of size $3$ in $M\setminus e$.
				Now by the inductive hypothesis,
				\begin{eqnarray*}
					[x^{k}]T_{M\setminus e}(x,0) &\le& \binom{n-k-3}{r-k}+ \binom{n-c_{3}+d_{3}(e)-k-3}{r-k-1}+\sigma,
				\end{eqnarray*}
				where $\sigma=t$ if $n\leq r+c_{3}-d_{3}(e)$ and $k=r-2$, and $\sigma=0$ otherwise.
				
				Then by Lemma \ref{lem2.1} and the definition of the Tutte polynomial,
				\begin{eqnarray*}
					[x^{k}]T_{M}(x,0)&=&[x^{k}]T_{M\setminus e}(x,0) + [x^{k}]T_{M/e}(x,0)\\
					&\le & \binom{n-k-2-d_{3}(e)}{r-k-1}+\binom{n-k-3}{r-k}+ \binom{n-c_{3}+d_{3}(e)-k-3}{r-k-1}+\sigma\\
					&\le&\binom{n-k-3}{r-k} + \binom{n-k-3}{r-k-1}+
					\binom{n-c_3-k-2}{r-k-1}+\sigma\quad \quad \text{(by Lemma 2.1)}\\
					&=&\binom{n-k-2}{r-k}
					+
					\binom{n-c_3-k-2}{r-k-1}+\sigma,
				\end{eqnarray*}
				where $\sigma=t$ if $n\leq r+c_{3}-d_{3}(e)\leq r+c_{3}$ and $k=r-2$, and $\sigma=0$ otherwise.
				
				Hence, 	the conclusion holds for $t=0$.
				
				(c) Note that $T_{M}(x,t)$ is a polynomial with non-negative coefficients.  For each $k$, the function $t \mapsto [x^k]T_M(x,t)$ is convex on $[0,1]$. Hence,  
				\begin{eqnarray*}
					[x^{k}]T_{M}(x,t)&\le&(1-t)[x^{k}]T_{M}(x,0) + t[x^{k}]T_{M}(x,1)\\
					&\le & (1-t)\left(\binom{n-k-2}{r-k}
					+
					\binom{n-c_3-k-2}{r-k-1}\right)\\
					&&+t\left(
					\binom{n-k-1}{r-k}
					-
					\binom{n-k-3}{r-k-2}
					+
					\binom{n-c_3-k-3}{r-k-2}
					+\sigma'\right)\\
					&=&	\binom{n-k-1}{r-k}
					-
					\binom{n-k-3}{r-k-2}
					+
					\binom{n-c_3-k-3}{r-k-2}
					\\
					&&+(1-t)
					\left[
					\binom{n-c_3-k-3}{r-k-1}
					-
					\binom{n-k-3}{r-k-1}
					\right]
					+t\sigma',
				\end{eqnarray*}
				where $\sigma'=1$ if $n\leq r+c_g$ and $k=r-g+1$, and $\sigma'=0$ otherwise. Hence, the conclusion holds for all admissible matroids of girth $3$, completing the proof of the base case.

				In conclusion, the result  holds.
				This completes the proof of Theorem \ref{thm1.6}.

				\section{Applications of main results}
				The following conclusion is an immediate consequence of Theorem \ref{thm1.6} (i).
				\begin{cor}\label{x-0}
					Let $M$ be a simple matroid of rank $r$ on $n$ elements with $1\le r\le n-1$, girth $g$ and $U_{g-1,g+1}$-free. Let $c_{g}$ denote the number of circuits of size $g$.  Then, for every non-negative integer $k$, 
					\begin{eqnarray*}
						[x^{k}]T_{M}(x,0) &\le& \binom{n-k-1}{r-k} - \binom{n-k-g+1}{r-k-g+2} + \binom{n-c_{g}-k-g+1}{r-k-g+2}+\sigma,
					\end{eqnarray*}
					where $\sigma =1$ if $n\le r+c_{g}$ and $k=r-g+2$, and $\sigma =0$ otherwise.	
				\end{cor}
				To obtain a corresponding upper bound on the coefficients of the characteristic polynomial, the following identity is required. It is not new and can be derived from Equation (3b) on page 8 in \cite{Riordan}.
				
				\begin{lem}\label{lem 2.7}\cite{Riordan}
					Let $m$ be a non-negative integer. Then, for any non-negative integers $q$ and $k$ with $k\le q$ and $m\ge q+1$,$$\sum_{i=0}^{q} \binom{m-i-1}{q-i} \binom{i}{k} = \binom{m}{q-k}.$$
				\end{lem}
				
				\begin{thm} \label{3.3}
					Let $M$ be a simple matroid of rank $r$ on $n$ elements with $1\le r\le n-1$, girth $g$ and $U_{g-1,g+1}$-free. Let $c_{g}$ denote the number of circuits of size $g$. Then, for every non-negative integer $k$,
					$$ (-1)^{r-k}[x^{k}]\chi_{M}(x)\le \binom{n}{r-k}-\binom{n-g+2}{r-k-g+2}+\binom{n-c_{g}-g+2}{r-k-g+2}+\psi, $$
					where $\psi =\binom{r-g+2}{k}-\binom{n-c_{g}-g+2}{r-k-g+2}$ if $n\le r+c_{g}$ and $0\le k\le r-g+2$, and $\psi =0$ otherwise.	
					
				\end{thm}
				
				\begin{proof}
					Let $[x^{k}]T_{M}(x,0)=a_{k}$. Then
					\[
					T_M(1-x,0) = \sum_{i=0}^{r} a_i (1-x)^i = \sum_{i=0}^{r} a_i \sum_{j=0}^i \binom{i}{j} (-x)^j = \sum_{j=0}^{r} (-1)^j x^j \sum_{i=j}^{r} a_i \binom{i}{j}.
					\]
					This implies that
					\[
					[x^k]T_M(1-x,0) = (-1)^k \sum_{i=k}^{r} a_i \binom{i}{k}.
					\]
					By Corollary  \ref{x-0}, $$a_{i}\le \binom{n-i-1}{r-i} - \binom{n-i-g+1}{r-i-g+2} + \binom{n-c_{g}-i-g+1}{r-i-g+2}+ \sigma_i,$$ 	where $\sigma_i=1$ if $n\leq r+c_{g}$ and $i=r-g+2$, and $\sigma_i =0$ otherwise.
					By Lemma \ref{lem 2.7},
					\begin{eqnarray*}
						&&(-1)^{r-k}[x^{k}]\chi_M(x)\\
						&\le& \sum_{i=k}^{r} \left[\binom{n-i-1}{r-i}-\binom{n-i-g+1}{r-i-g+2}+\binom{n-c_{g}-i-g+1}{r-i-g+2} +\sigma_i\right] \binom{i}{k}\\
						&=&\sum_{i=0}^{r} \binom{n-i-1}{r-i} \binom{i}{k} -\sum_{i=0}^{r-g+2} \binom{n-i-g+1}{r-i-g+2} \binom{i}{k}\\
						&&+\sum_{i=0}^{r-g+2} \binom{n-c_{g}-i-g+1}{r-i-g+2} \binom{i}{k}+\sum_{i=k}^{r}\sigma_i\binom{i}{k}.
					\end{eqnarray*}
					
					If	 $n> r+c_{g}$ or $k>r-g+2$, then $\sigma_i =0$ for any $i\ge k$. By Lemma \ref{lem 2.7}, $(-1)^{r-k}[x^{k}]\chi_M(x)\leq \binom{n}{r-k}-\binom{n-g+2}{r-k-g+2}+\binom{n-c_{g}-g+2}{r-k-g+2}$ and \(\psi=0\) in this case.
					
					If $n\le r+c_g$ and $k\le r-g+2$, by the definition of $\sigma_i$, only the term $i=r-g+2$ contributes. Hence, $	\sum_{i=k}^{r}\sigma_i\binom{i}{k} = \binom{r-g+2}{k}.$ Furthermore, for every $0\le i\le r-g+2$, we have $n-c_g-i-g+1 \le r-i-g+1 <r-i-g+2$; thus $\binom{n-c_g-i-g+1}{r-i-g+2}=0$. It follows that $(-1)^{r-k}[x^k]\chi_M(x) \le \binom{n}{r-k} - \binom{n-g+2}{r-k-g+2} +\binom{r-g+2}{k}$.					
					In this case, the estimate obtained directly from Corollary \ref{x-0} is
						\[
						\begin{aligned}
							(-1)^{r-k}[x^k]\chi_M(x)
							\leq{}&
							\binom{n}{r-k}
							-\binom{n-g+2}{r-k-g+2}+
							\binom{n-c_g-g+2}{r-k-g+2}\\
							&+
							\binom{r-g+2}{k}
							-\binom{n-c_g-g+2}{r-k-g+2}.
						\end{aligned}
						\]
			By the definition of $\psi $, the conclusion holds.
				\end{proof}
				\textbf{Remark.} When $n>r+c_g$, Theorem \ref{thm1.6} implies Theorem \ref{thm1.4}, when $n<r+c_g$, Theorem \ref{thm1.6} and Theorem \ref{thm1.4} do not imply each other.
				
				 Accordingly, under the binomial convention adopted in this paper, we treat the case $m_2<0$ separately in Lemma \ref{lem2.1} and add the correction term $1$ in Corollary \ref{x-0} when $n\le r+c_g$ and $k=r-g+2$. Its contribution under the binomial transform gives the correction term $\psi$ in Theorem \ref{3.3}.

				Recently, Ma, Guan and Jin \cite{Ma} established the following result, which is consistent with Theorem \ref{thm1.6}.
				\begin{thm}\label{thm3.2}\cite{Ma}
					Let $M$ be a matroid of rank $r$ on $n$ elements. If $ r\le n-1 $, then for any real number $t$ with $0\leq t\leq 1$, there exist integers $k'$, $k''$ such that 
					\item (i) for any $k>k'$,
					\begin{eqnarray*}
						[x^{k}]T_{M}(x,t)&=&\binom{n-k-1}{r-k}-\sum_{j=0}^{r-k} \binom{n-k-j-1}{r-k-j}|C_{j}|+ (t-1)\sum_{j=0}^{r-k+1} \binom{n-k-j-1}{r-k-j+1}|C_{j}|,
					\end{eqnarray*}
					where $C_{j}$ is the set of circuits of size $j$;
					\item (ii) for any $k>k''$,
					\begin{eqnarray*}
						[y^{k}]T_{M}(t,y)=\binom{n-k-1}{r-1}-\sum_{j=r+k}^{n} \binom{j-k-1}{r-1}|H_{j}|+ (t-1)\sum_{j=r+k-1}^{n} \binom{j-k-1}{r-2}|H_{j}|,
					\end{eqnarray*}
					where $H_{j}$ is the set of hyperplanes of size $j$.
				\end{thm}
				\begin{proof}[Comparison with Theorem \ref{thm1.6}(i)]
					For non-negative integers $a,b$, a positive integer $c_g$, and $0\le i<c_g$, we have $$\sum_{i=0}^{c_{g}-1} \binom{a-i}{b}=\binom{a+1}{b+1}-\binom{a+1-c_{g}}{b+1} \text{ and }  \binom{a}{b}\ge \binom{a-i}{b}.$$  Hence, $$ c_{g}\binom{a}{b} \ge \sum_{i=0}^{c_{g}-1}\binom{a-i}{b}. $$  For $k'<k\le r-g$, since $C_j=\varnothing$ for every $j<g$, 
					Theorem \ref{thm3.2}(i) gives
					\begin{eqnarray*}
						[x^{k}]T_{M}(x,t)&=&\binom{n-k-1}{r-k}-\sum_{j=g}^{r-k} \binom{n-k-j-1}{r-k-j}|C_{j}|- (1-t)\sum_{j=g}^{r-k+1} \binom{n-k-j-1}{r-k-j+1}|C_{j}|\\
						&\le &\binom{n-k-1}{r-k}-\binom{n-k-g-1}{r-k-g}|C_{g}|- (1-t) \binom{n-k-g-1}{r-k-g+1}|C_{g}|\\
						&\le &\binom{n-k-1}{r-k}-\sum_{i=0}^{c_{g}-1} \binom{n-k-g-1-i}{r-k-g}- (1-t)\sum_{i=0}^{c_{g}-1} \binom{n-k-g-1-i}{r-k-g+1}\\
						&=&\binom{n-k-1}{r-k}-\binom{n-k-g}{r-k-g+1}+\binom{n-k-g-c_{g}}{r-k-g+1}\\&&+(1-t)\left[\binom{n-k-g-c_{g}}{r-k-g+2}-\binom{n-k-g}{r-k-g+2}\right].
					\end{eqnarray*}
					Thus, for $k'<k\le r-g$, Theorem \ref{thm3.2}(i) yields the upper bound in Theorem \ref{thm1.6}(i).
				\end{proof}

				It is easy to see that for any $E'\subseteq E$, $E-E'$ is a bond (i.e., minimal edge cut) of $G$ if and only if $E'$ is a hyperplane in $M(G)$, and $E'$ is the edge set of a cycle of $G$ if and only if
				$E'$ is a circuit of $M(G)$. Note that $M(G)$ must be $U_{g-1,g+1}$-free when $G$ is simple with girth $g$. Thus, the following two corollaries follow from
				Theorems \ref{thm1.3}, \ref{thm1.6} and \ref{thm1.7}.

				\begin{cor}
					Let $G=(V,E)$ be a connected simple graph with girth $g$, edge connectivity $\kappa'\ge 3$,  and $|V|\leq |E|-1$. Let $M=M(G)$, and let $a$ and $c_{g}$ denote the number of $\kappa'$-edge cuts and the number of cycles of length $g$, respectively. Then for any non-negative integer $k$ and any real number $t$ with $0\leq t\leq 1$,
					\begin{flalign*}
						&(i)\qquad
						&[x^k]T_M(x,t)
						\leq&
						\binom{|E|-k-1}{|V|-1-k}
						-
						\binom{|E|-k-g}{|V|-k-g}
						+
						\binom{|E|-c_g-k-g}{|V|-k-g}
						&&\\
						&&
						& +(1-t)
						\left[
						\binom{|E|-c_g-k-g}{|V|-k-g+1}
						-
						\binom{|E|-k-g}{|V|-k-g+1}
						\right]
						+\sigma,
						&&
					\end{flalign*}
					where $\sigma=t$ if $|E|\le |V|-1+c_{g}$ and  $k=|V|-g$, $\sigma=1-t$ if $|E|\le |V|+c_{g}$ and  $k=|V|-g+1$, and $\sigma=0$ otherwise.
					\begin{flalign*}
						&(ii)\qquad
						&[y^k]T_M(t,y)
						\leq&
						\binom{|E|-k-1}{|V|-2}
						-
						\binom{|E|-k-\kappa'}{|V|-2}
						+
						\binom{|E|-k-\kappa'-a}{|V|-2-a}
						&&\\
						&&
						& +(1-t)
						\left[
						\binom{|E|-k-\kappa'-a}{|V|-3-a}
						-
						\binom{|E|-k-\kappa'}{|V|-3}
						\right]
						+\theta,
						&&
					\end{flalign*}
					where $\theta=t$ if $|V|\le a+1$ and  $k=|E|-|V|-\kappa'+2$, $\theta=1-t$ if $|V|\le a+2$ and  $k=|E|-|V|-\kappa'+3$, and $\theta=0$ otherwise.
				\end{cor}
				
				\begin{cor}
					Let $G=(V,E)$ be a connected graph with at least two vertices, girth $g$, edge connectivity $\kappa'$,  and $|V|\leq |E|-1$. Let $M=M(G)$, and let $a$ and $c_{g}$ be the number of $\kappa'$-edge cuts and the number of cycles of length $g$, respectively. Then for any non-negative integer $k$,
					\begin{flalign*}
						&(i)\qquad
						&&
						[x^k]T_M(x,1) 
						\leq
						\binom{|E|-k-1}{|V|-1-k}
						-
						\binom{|E|-k-g}{|V|-k-g}
						+
						\binom{|E|-c_g-k-g}{|V|-k-g}
						+\sigma,
						&&
					\end{flalign*}
					where $\sigma=1$ if $|E|\le |V|-1+c_g$ and $k=|V|-g$, and $\sigma=0$ otherwise.
					\begin{flalign*}
						&(ii)\qquad 
						&&
						[y^k]T_M(1,y)
						\leq
						\binom{|E|-k-1}{|V|-2}
						-
						\binom{|E|-k-\kappa'}{|V|-2}
						+
						\binom{|E|-k-\kappa'-a}{|V|-2-a}
						+\theta,
						&&
					\end{flalign*}
					where $\theta=1$ if $|V|\le a+1$ and $k=|E|-\kappa'-|V|+2$, and $\theta=0$ otherwise.
				\end{cor}

				\section{Conclusions and Problems}
				In this paper, we study the upper bounding of the coefficients for polynomials $T_M(x,t)$, where $0\le t\le 1$ and $M$ is a simple matroid with girth $g$ that is $U_{g-1,g+1}$-free.	The assumption \(0\leq t\leq 1\) cannot be replaced by \(t>1\) or \( t<0 \).
				\begin{exm}
					Let \(M=M(K_4)\). Then $n=6$, $r=3$, $g=3$, $c_3=4$, and $T_M(x,t)=x^3+3x^2+2x+4xt+t^3+3t^2+2t$. For $k=1$, the right-hand side of Theorem \ref{thm1.6}(i) is $3+3t$, whereas $[x^{1}]T_M(x,t)=2+4t$. Thus the bound fails for every $t>1$. For $k=2$, that right-hand side is $3+t$, whereas $[x^2]T_M(x,t)=3$, so the bound fails for every $t<0$.
				\end{exm} The preceding results motivate the following question.
				
				\begin{que}
					Which classes of simple nonuniform matroids $M$ of girth $g$
					that contain a $U_{g-1,g+1}$-restriction satisfy the bound
					in Theorem \ref{thm1.6}(i) for every non-negative integer $k$
					and every $0\le t\le 1$?
				\end{que}
				
				\section*{Use of generative AI}
				
				During the preparation of this work, the author used ChatGPT to correct grammatical errors, refine wording, and improve sentence structure and clarity, as well as to assist with \LaTeX{} formatting. After using this tool, the authors reviewed and edited the relevant content and take full responsibility for the final manuscript.

				\section*{Acknowledgements}
				
			This work was supported by National Natural Science Foundation of China (Nos.~12401462 and 12371356), the Natural Science Foundation of Shanxi Province (No.~202403021222034), and the Shanxi Key Laboratory of Digital Design and Manufacturing (No. 202204010931025).
				
			\end{document}